\documentclass[12pt]{amsart}
\newtheorem{theorem}{Theorem}[section]
\newtheorem{lemma}[theorem]{Lemma}
\newtheorem{proposition}[theorem]{Proposition}
\newtheorem{corollary}[theorem]{Corollary}

\theoremstyle{definition}

\newtheorem{remark}[theorem]{Remark}

\numberwithin{equation}{section}

\begin{document}

\baselineskip=15.5pt

\title[On a smooth compactification of
${\rm PSL}(n, \mathbb C)/T$]{On a smooth compactification of
${\rm PSL}(n, \mathbb C)/T$}

\author[I. Biswas]{Indranil Biswas}

\address{School of Mathematics, Tata Institute of Fundamental
Research, Homi Bhabha Road, Mumbai 400005, India}

\email{indranil@math.tifr.res.in}

\author[S. S. Kannan]{S. Senthamarai Kannan}

\address{Chennai Mathematical Institute, H1, SIPCOT IT Park, Siruseri,
Kelambakkam 603103, India}

\email{kannan@cmi.ac.in}

\author[D. S. Nagaraj]{D. S. Nagaraj}

\address{The Institute of Mathematical Sciences, CIT
Campus, Taramani, Chennai 600113, India}

\email{dsn@imsc.res.in}



\begin{abstract}
Let $T$ be a maximal torus of ${\rm PSL}(n, \mathbb C)$. For $n\,\geq\, 4$, we
construct a smooth compactification of ${\rm PSL}(n, \mathbb C)/T$ as a geometric
invariant theoretic quotient
of the wonderful compactification $\overline{{\rm PSL}(n, \mathbb C)}$ for a suitable
choice of $T$--linearized ample line bundle on $\overline{{\rm PSL}(n, \mathbb C)}$. We also prove
that the connected component, containing the identity element, of the automorphism group
of this compactification of ${\rm PSL}(n, \mathbb C)/T$ is ${\rm PSL}(n, \mathbb C)$ itself.
\end{abstract}

\maketitle

\section{Introduction}

Let $G$ be a semisimple group of adjoint type  over the field $\mathbb{C}$ of
complex numbers. De Concini and Procesi in \cite{DP} constructed a smooth projective 
variety $\overline{G}$ with an action of $G\times G$ such that
\begin{itemize}
\item the variety $G$ equipped with the action of $G\times G$
given by the left and right  translations is an open dense orbit of it, and

\item the boundary $\overline{G}\setminus G$ is a union of 
$G\times G$ stable normal crossing divisors.
\end{itemize}
This variety $\overline{G}$ is known as the wonderful compactification
of $G$.

Fix a maximal torus $T$ of $G$.
Consider the right action of $T$ on $\overline{G}$,
meaning the action of the subgroup 
$1\times T\, \subset\, G \times G$. For a $T$--linearized ample
line bundle $\mathcal{L}$ on $\overline{G}$, let 
$\overline{G}_{T}^{ss}(\mathcal{L})$ and 
$\overline{G}_{T}^{s}(\mathcal{L})$ denote respectively the loci of semistable
and stable points of $\overline{G}$ (see \cite[p. 30, p. 40]{MFK}). 

Our first main result (Proposition \ref{prop1}) says that there is a $T$--linearized
ample line bundle $\mathcal{L}$ on $\overline{G}$ such that $\overline{G}_{T}^{ss}
(\mathcal{L}) \,=\,\overline{G}_{T}^{s}(\mathcal{L})$.

For $G\,=\, {\rm PSL}(n, \mathbb C)$, we show that
there is a $T$--linearized ample line
bundle $\mathcal{L}$ on $\overline{{\rm PSL}(n, \mathbb C)}$ such that
\begin{itemize}
\item the GIT quotient $\overline{{\rm PSL}(n,
\mathbb C)}_{T}^{ss}(\mathcal{L})/\!\!/T$ is smooth, and 

\item the boundary $(\overline{{\rm PSL}(n, \mathbb C)}_{T}^{ss}
(\mathcal{L})/\!\!/T)\setminus ({\rm PSL}(n, \mathbb C)/T)$
is a union of ${\rm PSL}(n, \mathbb C)$
stable normal crossing divisors.
\end{itemize}

We further show that for $n\,\geq\, 4$, the connected component of the 
automorphism group of $\overline{{\rm PSL}(n, \mathbb C)}_{T}^{ss}/\!\!/T$
containing the identity automorphism is ${\rm PSL}(n, \mathbb C)$ (Theorem
\ref{thm1}).

\section{Preliminaries and notation}\label{se2}

In this section we recall some preliminaries and notation about Lie algebras and
algebraic groups; see for example \cite{Hu} and \cite{Hu1} for the details.
Let $G$ be a simple group of adjoint type  of rank $n$
over the field of complex numbers. Let $T$ be a maximal torus of $G$ and
$B\,\supset\, T$ a Borel subgroup of $G.$ 
Let $N_{G}(T)$ denote the normalizer of $T$ in $G.$ So $W\,:=\,N_{G}(T)/T$
is the Weyl group of $G$ with respect to $T.$

The Lie algebra of $G$ will be denoted by $\mathfrak{g}.$ Let 
$\mathfrak{h}\,\subset\, \mathfrak{g}$ be the Lie algebra of $T.$
The set of roots of $G$ with respect to $T$ will be denoted by $R.$
Let $R^{+}\,\subset\, R$ be the set of positive roots with respect to $B.$
Let $$S\,=\,\{\alpha_1\, , \alpha_2\, , \cdots \, ,\alpha_n\}\,\subset\, R^{+}$$
be the set of simple roots with respect to $B.$ The group of characters of $T$ will
be denoted by $X(T),$ while the group of one-parameter 
subgroups of $T$ will be denoted by $Y(T).$ Let 
$$\{\lambda_i\, \mid\, 1\,\leq\, i \,\leq\, n\}$$ be the ordered set of one-parameter
subgroups of $T$ satisfying the condition that $\langle \alpha_i\, , \lambda_j\rangle
\,=\, \delta_{ij},$ where
$$\langle -\, ,- \rangle \,:\, X(T)\times Y(T) \,\longrightarrow \,\mathbb{Z}$$
is the natural pairing, and $\delta_{ij}$ is the Kronecker delta function.
Let $\leq$ (respectively, $\geq$) be the partial order on $X(T)$ 
defined as follows:
$\chi_1\leq\chi_2,$ (respectively, $\chi_1\geq\chi_2$) if $\chi_2-\chi_1$ 
(respectively, $\chi_1-\chi_2$) is a linear combination of simple roots
with non-negative integers as coefficients.

Let $( - \, , - )$ denote the restriction of the Killing form of 
$\mathfrak{g}$ to 
$\mathfrak{h}.$ Let $$\{\omega_j\,\mid\, 1\,\leq\, j\,\leq \,n\}$$ 
be the ordered set of fundamental weights 
corresponding to $S,$ in other words, $$\frac{2(\omega_i,\alpha_j)}{(\alpha_j,\alpha_j)}
\,=\, \delta_{ij}\, , ~\ 1 \,\leq\, i\, ,j \,\leq\, n\, .$$
For $1 \,\leq\, i \,\leq\, n$, let $s_{\alpha_i}$ denote the simple reflection
corresponding to $\alpha_{i}$.

The longest element of $W$ corresponding to $B$ will be
denoted by $w_0$. Let $$B^-\,=\,w_0Bw_0^{-1}$$ be the Borel subgroup of $G$
opposite to $B$ with respect to $T$.

For the notion of a $G$--linearization, and the GIT quotients, 
we refer to \cite[p. 30, p. 40]{MFK}. 

Consider the flag variety $G/B$ that parametrizes all Borel subgroups of $G$. For
a character $\chi$ of $B,$ let 
$${L}_{\chi}\,=\, G\times_B{\mathbb C}
\,\longrightarrow\, G/B$$ be the
$G$--linearized line bundle associated to the action of $B$ on 
$G\times{\mathbb C}$ given by $b.(g,z)=(gb,\chi(b^{-1})z)$ for $b\in B$ and
$(g,z)\in G\times{\mathbb C}.$ So, in particular, ${L}_{\chi}$
is $T$--linearized. When ${L}_{\chi}$ is ample, we denote by $(G/B)^{ss}_T
(L_\chi)$ (respectively, $(G/B)^{s}_T(L_\chi)$) the
semistable (respectively, stable) locus in $G/B$ for the $T$--linearized 
ample line bundle ${L}_\chi$.

Next we recall some facts about the wonderful compactification of $G.$
Let $\chi$ be a regular dominant weight of $G$ with respect to $T$ and $B$,
and let $V({\chi})$ be the irreducible representation of $\widehat{G}$ with 
highest weight $\chi,$ where $\widehat{G}$ is the simply connected covering 
of $G.$ By \cite[p. 16, 3.4]{DP}, the wonderful compactification
$\overline{G},$ which we denote by $X,$ is the closure of the 
$G\times G$--orbit of the point 
$$[1]\,\in\, \mathbb{P}(V({\chi})\otimes V({\chi})^{*})$$ corresponding to the 
identity element $1$ of $V({\chi})\otimes V({\chi})^{*}\,=\,
{\rm End}(V({\chi})^{*}).$ We denote by $\mathcal{L}_{\chi}$ the ample line
bundle on $X$ induced by this projective embedding. Since the
regular dominant weights generate the weight lattice, given a weight 
$\chi,$ we have the line bundle $\mathcal{L}_{\chi}$ on $X$ 
associated to $\chi.$

By \cite[Theorem, p. 14, Section 3.1]{DP}, there is a  unique closed $G\times G$--orbit
$Z$ in $X.$ Note that $$Z\,=\, \bigcap_{i=1}^{n}D_{i}\, ,$$ where $D_{i}$ is the
$G\times G$ stable irreducible
component of $\overline{G}\setminus G$ such that $\mathcal{O}(D_{i})=\mathcal{L}_{\alpha_{i}}$
\cite[p. 29, Section 8.2, Corollary]{DP}.
Further, $Z$ is isomorphic to
$G/B\times G/{B^{-}}$ as a $G\times G$ variety.
By \cite[p. 26, 8.1]{DP}, the pullback homomorphism $$i^{*}
\,:\,{\rm Pic}(X)\,\longrightarrow\, {\rm Pic}(Z)\, ,$$
for the inclusion map $i:Z\hookrightarrow X$ is injective and is given by
$$i^{*}(\mathcal{L}_{\chi})= p^{*}_1({L}_{\chi})\otimes 
p^{*}_2({L}_{-\chi}),$$
where ${L}_{\chi}$ (respectively, ${L}_{-\chi}$) is
the line bundle on $G/B$ (respectively, $G/{B^{-}}$) associated
$\chi$ (respectively, $-\chi$) and $p_j$ is
the projection to the $j$-th factor of $G/B\times G/B^{-1}$ for $j\,=\,1,2.$

\section{Choice of a polarization on $\overline{G}$}

We continue with the notation of Section \ref{se2}. Let $G$ be a simple algebraic group of 
adjoint type of rank $n\geq 2,$ such that its root system $R$ is different
from $A_2.$ 
Let 
$$\mathbb{N}S:=\{\sum_{i=1}^{n}m_{i}\alpha_{i}: m_{i}\in \mathbb{N}\}\, .$$ 

Then, we have the following:

\begin{lemma}\label{lem1}
The above defined 
$\mathbb{N}S$ contains a regular dominant character $\chi$ of $T$ such 
that $s_{\alpha_{i}}(\chi)\geq 0$ and $\langle \chi , w(\lambda_{i}) \rangle 
\neq 0$ for every $w\in W$ and 
$1\leq i \leq n.$  
\end{lemma}

\begin{proof}
Denote by $X(T)_{\mathbb{Q}}$ the rational vector space generated by $X(T)$,
and also denote by $X(T)^{+}$ the semi-group of it given by the dominant characters of $T.$ 
Let $\rho\in X(T)_{\mathbb{Q}}$ be the half sum of positive roots of $R.$ 
Then, $2\rho\,=\,2(\sum_{i=1}^{n}\omega_{i})\in X(T)^{+}$ is a regular
dominant character of $T$, and we have $2\rho \in \mathbb{N}S.$ 

Since $R$ is irreducible of rank at-least $2$ and different from $A_{2}$, we see 
that for every simple root $\alpha_{i}$, there are at-least $3$ positive roots 
$\beta$ satisfying $\alpha_{i}\leq \beta.$ Hence, the coefficient of every simple root $\alpha_{j}$ in the expression of 
$s_{\alpha_{i}}(2\rho)=2\rho-2\alpha_{i}$ (as a non-negative integral linear 
combination of simple roots) is positive. Hence, 
we have $s_{\alpha_{i}}(2\rho)\in \mathbb{N}S.$ Thus, we have 
$$2\rho\,\in\, X(T)^{+}\cap(\bigcap_{i=1}^{n}s_{\alpha_{i}}(\mathbb{N}S))\, .$$

Denote by $N$ the determinant of the Cartan matrix of $R$. Then we have 
$N\omega_{i}\in \mathbb{N}S$ for every $i=1, 2, \cdots , n.$ By the previous 
discussion, there exists $m\in \mathbb{N}$ such that 
$ms_{\alpha_{i}}(2\rho)-N\alpha_{i}\in \mathbb{N}S$ for every $1\leq i\leq n.$
Hence, we get 
$$s_{\alpha_{i}}(2m\rho+N\omega_{i})=ms_{\alpha_{i}}(2\rho)-
N\alpha_{i}+N\omega_{i}\in \mathbb{N}S,\,\,1\leq i\leq n\, ,$$  
and from this it follows that 
$$2m\rho+N\omega_{i}\in X(T)^{+}\cap(\bigcap_{j=1}^{n}s_{\alpha_{j}}(\mathbb{N}S)), \,\,1\leq i \leq n.$$

Consider the characters $2m\rho,$ $2m\rho+N\omega_{2},\cdots ,$  
$2m\rho+N\omega_{n}$ of $T.$ These are linearly independent in 
$X(T)$ and by construction they all lie in the  
rational cone $$\mathcal{C}\,\subset\, X(T)_{\mathbb{Q}}^{+}$$
generated by the semi-group $X(T)^{+}\bigcap (\bigcap_{i=1}^{n}s_{\alpha_{i}}(\mathbb{N}S)).$
It follows that $\mathcal{C}$ has a maximal dimension in $X(T)_{\mathbb Q},$
hence it is not contained in any hyperplane of $X(T)_{\mathbb Q}.$ Therefore,
there exists a regular dominant character 
$\chi\in \mathcal{C}\bigcap \mathbb{N}S$
such that $\langle \chi , w(\lambda_{i}) \rangle \neq 0$ for all 
$1\leq i \leq n$ and every $w\in W,$ and hence the lemma follows. 
\end{proof}

\begin{lemma}\label{lem2}
Let $\chi \in \mathbb{N}S$ be a regular dominant character of $T$ 
satisfying the properties stated in Lemma \ref{lem1}. Then we have 
\begin{enumerate} 
\item $(G/B)^{ss}_T( L_\chi)\,=\, (G/B)^{s}_T( L_\chi)$, and 

\item the set of all unstable points 
$$(G/B)\setminus (G/B)_{T}^{ss}( L_\chi)$$
is contained in the union of $W$--translates of all Schubert varieties of 
codimension at least two.
\end{enumerate}
\end{lemma}

\begin{proof} 
Set $L\,:=\, L_{\chi}.$ Since $\langle \chi , w(\lambda_{i}) \rangle 
\neq 0$ for every $w\in W$ and 
$1\leq i \leq n,$ by
\cite[p. 38, Lemma 4.1]{Ka1} we have
$$(G/B)^{ss}_T( L)\,=\, (G/B)^{s}_T(L)\,.$$ 
This proves (1). 

To prove (2), take an unstable point $x\,\in\, G/B$
for the polarization $L.$ Then, there is a one-parameter 
subgroup $\lambda$ of $T$ such that $\mu^{{L}}(x, \lambda)\,<\, 0$.
Let $\phi\,\in\, W$ be such that $\phi(\lambda)$ is in the fundamental chamber, 
say $$\phi(\lambda) \,= \,\sum_{i=1}^n c_i \lambda_i\, ,$$ where $\{c_i\}$ are
non-negative integers. Consequently, we have 
$$\mu^{{L}}(n_{\phi}(x), \phi(\lambda)) 
\,=\, \mu^{L}(x,\lambda)\,< \,0\, ,$$
where $n_{\phi}$ is a representative of $\phi$ in $N_G(T).$
Now, let $n_{\phi}(x)$ be in the Schubert cell $BwB/B$ for some 
$w \,\in\, W.$
By \cite[Lemma 5.1]{Se}, we have
$$ \mu^{{L}}(n_{\phi}(x), \phi(\lambda)) \,=\, 
(-\sum_{i=1}^nc_i \langle w(\chi)\, ,\lambda_i\rangle)
 \,<\, 0\, .$$
(The sign here is negative because we are using
left action of $B$ on $G/B$  while in \cite[Lemma 5.1]{Se} the action of 
$B$ on $B\backslash G$ is on the right.)
Therefore we have $w(\chi)\,\not\leq\, 0.$ For every $1\leq i \leq n$ we have  
$s_{\alpha_{i}}(\chi)\,\geq\, 0$, and hence 
$w_0 s_{\alpha_{i}}(\chi)\,\leq\, 0$. Hence we have $l(w_{0})-l(w)\,\geq\, 2.$
This completes the proof of (2).
\end{proof}

\begin{proposition}\label{prop1}
Let $X\,=\,\overline{G}$ be the
wonderful compactification of $G.$ Let $\chi$ be as in Lemma \ref{lem2}, and
let $X^{ss}_T(\mathcal L_{\chi})$ (respectively, $X^{s}_T(\mathcal L_{\chi})$) 
be the semi-stable (respectively, stable) locus of $X$ for the action of $1\times T$
and the polarization $\mathcal L_{\chi}$ on $X$. Then we have
\begin{enumerate}
\item $X^{ss}_T(\mathcal L_{\chi})\,=\, X^{s}_T(\mathcal L_{\chi})$, and 

\item the set of unstable points $X\setminus (X^{ss}_T(\mathcal L_{\chi}))$ is a
union of irreducible closed subvarieties of codimension at least three.
\end{enumerate}
\end{proposition}

\begin{proof}
Let $Z$ be the unique closed $G\times G$--orbit in 
$X.$  Let $Z^{ss}_T(\mathcal L_{\chi})$ (respectively, 
$Z^{s}_T(\mathcal L_{\chi})$) 
be the semi-stable (respectively, stable) locus of $Z$ for the action of 
$1\times T$ and the polarization $i^{*}(\mathcal{L}_{\chi}),$ where
$i\,:\,Z \,\hookrightarrow \,X$ is the inclusion map.
Since $Z$ is isomorphic to $G/B\times G/B^{-}$ and 
$i^{*}(\mathcal{L}_{\chi})=p^{*}_1({L}_{\chi})\otimes p^{*}_2({L}_{-\chi})$,
we see that 
$$Z_{T}^{ss}(\mathcal{L}_{\chi})\simeq (G/B)\times ((G/{B^{-}})^{ss}_{T}
({L}_{-\chi}))$$ and $Z_{T}^{s}(\mathcal{L}_{\chi})\simeq (G/B)\times ((G/{B^{-}})^{s}_{T}
({L}_{-\chi})).$
Set $Z^{ss}\, =\, Z^{ss}_T(\mathcal L_{\chi})$
and $Z^{s}\,=\, Z^{s}_T(\mathcal L_{\chi})$. By Lemma \ref{lem2} and above discussion, we have 
$Z^{ss}\,=\, Z^{s}.$ 

For convenience, we will denote $X^{ss}_T(\mathcal L_{\chi})$ 
and $X^{s}_T(\mathcal L_{\chi})$
by $X^{ss}$ and $X^s$ respectively. If $X^{ss}\,\neq\, X^{s},$ 
then the complement $X^{ss} \setminus X^{s}$ is a non-empty $G\times T$ invariant closed 
subset of $X^{ss}$. Hence, the complement $(X^{ss}/\!\!/T) \setminus (X^{s}/\!\!/T)$ is a
non-empty $G\times\{1\}$--invariant closed subset of $X^{ss}/\!\!/T$. In particular, $(X^{ss}/\!\!/T)
\setminus (X^{s}/\!\!/T)$ is a finite union of
non-empty $G\times\{1\}$--invariant projective varieties.
Therefore, there is a $B\times\{1\}$--fixed point
in $(X^{ss}/\!\!/T) \setminus (X^{s}/\!\!/T).$ Let
$$
p\in\, (X^{ss}/\!\!/T) \setminus (X^{s}/\!\!/T)
$$
be a $B\times \{1\}$--fixed point. Let $Y$ be the closed $\{1\}\times T$--orbit in the 
fiber $\pi^{-1}(\{p\})$ over $p$ for the geometric invariant theoretic 
quotient map $\pi:X^{ss}\,\longrightarrow \,X^{ss}/\!\!/T.$ Since this
map $\pi$ is $G\times \{1\}$ equivariant, we conclude that
$\pi^{-1}(\{p\})$ is $B\times \{1\}$--invariant. Hence, for 
any $b\in B,$ the translation $(b,1)\cdot Y$ lies in $\pi^{-1}(\{p\}).$ Since
the actions of $B\times \{1\}$ and $\{1\} \times T$ on $X$ commute with each other, we see that 
$(b,1)\cdot Y$ is also a closed $\{1\} \times T$--orbit in $\pi^{-1}(\{p\}).$ 
By the uniqueness of the closed $\{1\} \times T$--orbit in $\pi^{-1}(\{p\})$ 
we conclude that $(b,1)\cdot Y=Y.$ Hence $Y$ is preserved by the action of $B\times \{1\}$.
In particular, $Y$ is $U\times \{1\}$--invariant, where $U\,\subset\, B$ is 
the unipotent radical. The action of $U\times \{1\}$ on $Y$ induces a 
homomorphism from $U$ to $T/S$ of algebraic groups, where $\{1\} \times S$ is 
the stabilizer in $\{1\}\times T$ of some point $q$ in $Y$. Since there is no 
nontrivial homomorphism from an unipotent group to a torus, we conclude that 
$U\times \{1\}$ fixes the point $q.$

By \cite[p. 32, Proposition]{DP}, for any regular dominant character $\chi$ 
of $T$ with respect to $B,$ the morphism
$X\,\hookrightarrow\, \mathbb{P}(V(\chi)\otimes V(\chi)^{*})$ is a $G\times G$ equivariant embedding, where
$V(\chi)$ is the irreducible 
representation of $G$ with highest weight $\chi$, and $V(\chi)^{*}$ is its
dual. Hence, the $U\times 1$--fixed point set
of $X$ is equal to $X\bigcap \mathbb{P}(\mathbb{C}_{\chi}\otimes V(\chi)^{*}),$ where $\mathbb{C}_{\chi}$
is the one  dimensional $B$--module associated to the character $\chi.$
Therefore, by the above discussion,  
we have $q\in X\bigcap \mathbb{P}(\mathbb{C}_{\chi}\otimes V(\chi)^{*}).$

Further, by \cite[Theorem, p. 30]{DP} we have  
$$H^{0}(X, \mathcal{L}_{\chi})=\bigoplus_{\nu\leq \chi}V(\nu)^{*}\otimes
V(\nu),$$
where the sum runs over all dominant characters $\nu$ of $T$ satisfying $\nu\leq \chi.$
By \cite[p. 29, Corollary]{DP} and \cite[p. 30, Theorem]{DP}, the zero locus of
$$\bigoplus_{\nu < \chi}V(\nu)^{*}\otimes V(\nu)\,\subset\, H^{0}(X, \mathcal{L}_{\chi})$$
in $X$ is the unique closed 
$G\times G$--orbit $Z=G/B\times G/B^{-}.$  Hence, by the discussion in the previous paragraph,
we have  $q\,\in\, Z.$ This contradicts the choice of the 
polarization $\mathcal L_{\chi}.$ Therefore, the proof of (1) is complete.

To prove (2), note that $X \setminus X^{ss}$ is a closed subset of $X$,
and $$Z \setminus Z^{ss} \,=\, (X \setminus X^{ss})\cap Z\, .$$ Also, by 
Lemma \ref{lem2}, the complement $Z \setminus Z^{ss}\, \subset\, Z$ is of codimension
at least two. Since we have $Z\,=\,\bigcap_{i=1}^{n}D_{i}$, the complement $D_{i} \setminus
D_{i}^{ss}$ is of codimension 
at least two for all $1\,\leq\, i \,\leq\, n.$ Further, every point in the open subset 
$G\,\subset\,X$ is semistable. Hence, $X \setminus X^{ss}$ is of codimension
at least three.
\end{proof}

The following lemma will be used in the proof of Corollary \ref{cor1}. 

\begin{lemma}\label{lem3} Let $H$ be a reductive algebraic group acting
linearly on a polarized projective variety $V.$ Assume that $V^{ss}\,=\,
V^{s},$ where $V^{ss}$ (respectively, $ V^{s}$) is the set of semi-stable 
(respectively, stable) points of $V$ for the action of $H$.
Then the set of all points in $V^{ss}$ whose stabilizer in $H$ 
is trivial is actually a Zariski open subset (it may be possibly empty).
\end{lemma}
 
\begin{proof} 
Consider the morphism
$$f\,:\, H \times V^{ss} \, \longrightarrow\, V^{ss} \times V^{ss}\, ,~ \
(h\, , v) \, \longmapsto\, (h\cdot v\, , v)\, .$$
Since $V^{ss}\,=\,V^{s},$ this map $f$ is proper \cite[p. 55, Corollary 2.5]{MFK}.
Hence the image $$M\,:=\, f(H\times V^{ss})\,\subset\, V^{ss}\times V^{ss}$$ 
is a closed subvariety.
Now, let $$U'\,\subset\, V^{ss}$$ be the locus of points with trivial stabilizer
(for the action of $H$). Take any
$$
v_0\,\in\, U'\, ,
$$
and set $z_{0}\,:=\,f((1\, , v_{0}))\,=\,(v_{0}\, , v_{0}).$ Then, $(f_{*}
\mathcal{O}_{H\times V^{ss}})_{z_{0}}$ is a free $\mathcal{O}_{M, z_{0}}$--module of
rank one. Hence by 
\cite[p. 152, Souped-up version II of Nakayama's Lemma]{Mu},
the locus of points $x\,\in\, M$ such that $(f_{*}\mathcal{O}_{H\times V^{ss}})_{x}$
is a free $\mathcal{O}_{M, x}$--module of rank at most one is a
non-empty Zariski open subset. Since 
$(f_{*}(\mathcal{O}_{H\times V^{ss}}))_{z}$ is
nonzero for all $z\,\in\, M$, the set of all points $x\,\in\, M$ such that
$(f_{*}(\mathcal{O}_{H\times V^{ss}}))_{x}$ is a free $\mathcal{O}_{M, x}$--module of 
rank one is a Zariski open subset of $M$; this Zariski open subset of $M$
will be denoted by $U$. Note that $$f^{-1}(U)\,=\, p_2^{-1}(U')\, ,$$ 
where $p_2\,:\, H \times V^{ss}\,\longrightarrow\, V^{ss}$ is the second projection.
Since  $p_2$ is flat of finite type over
$\mathbb{C},$ it is an open map (see \cite[p. 266, Exercise 9.1]{Ha}). 
Hence $U'\,=\,p_{2}(f^{-1}(U))$ is a Zariski open subset.
This finishes the proof of the lemma.
\end{proof}

\begin{corollary}\label{cor1} Let $X\,=\,\overline{{\rm PSL}(n+1, \mathbb C)}$ be the wonderful compactification of
${\rm PSL}(n+1, \mathbb C),\,\, n\geq 3.$ For the choice of the regular dominant character $\chi$ of $T$ as in 
Proposition \ref{prop1},
\begin{enumerate}
\item the action of $\{1\}\times T$ on $X^{ss}_T(\mathcal L_{\chi})/\!\!/T$
is free,

\item $X^{ss}_T(\mathcal L_{\chi})/\!\!/T$ is a smooth projective embedding
of $G/T,$  and

\item the set of unstable points $X\setminus (X^{ss}_T(\mathcal L_{\chi}))$ is 
a union of irreducible closed subvarieties of codimension at least three. 
\end{enumerate}
\end{corollary}

\begin{proof}
Let $\chi$ be a regular dominant character of $T$ as in 
Proposition \ref{prop1}. As in the proof of Proposition \ref{prop1}, let
$Z$ denote the unique closed $G\times G$ orbit in $X.$ Also, let $X^{ss}$,
$X^s$, $Z^{ss}$ and $Z^s$ be as in the proof of 
Proposition \ref{prop1}.
 By Proposition \ref{prop1} we have $X^{ss}\,=\, X^s.$  Hence
by Lemma \ref{lem3}, the locus $V$ of points in $ X^{ss}$
 with trivial stabilizer (for the action of 
$\{1\}\times T$) is a Zariski open subset of $X^{ss}.$ Therefore,
$X^{ss}\setminus V$ is a $G\times \{1\}$ stable closed subvariety of $X^{ss}.$
By using the arguments in the the proof of Proposition \ref{prop1}
we see that the set of $B\times \{1\}$--fixed points in $Z \cap (X^{ss}\setminus V)$
is non-empty.
But on the other hand by the proof of \cite[p. 194, Example 3.3]{Ka2} we see that
given any point $z\,\in\, Z^{ss}$, its stabilizer subgroup in $\{1\}\times T$
is trivial. This is a contradiction. Hence we conclude that 
the action of $\{1\}\times T$ on $X^{ss}$ is free. This proves Part (1) and Part (2).

Part (3) follows immediately from the corresponding statement 
in Proposition \ref{prop1}.
\end{proof}

\section{Automorphism group of $\overline{{\rm PSL}(n+1, \mathbb C)}_{T}^{ss}
(\mathcal{L})/\!\!/T$}

Let $G\,=\,{\rm PSL}(n+1, \mathbb C)$, with $n\, \geq\, 3,$ and define
$$Y:=\overline{{\rm PSL}(n+1, \mathbb C)}_{T}^{ss}(\mathcal 
L_{\chi})/\!\!/T\, ,$$ where $\chi$ is as in Proposition \ref{prop1}.

\begin{theorem}\label{thm1} Let $A$ denote the connected component,
containing the identity element, of the group of holomorphic (= algebraic)
automorphisms of $Y.$ Then
\begin{enumerate}
\item $A$ is isomorphic to $G$, and 
\item the Picard group of $Y$ is a free Abelian group of rank $2n$.
\end{enumerate}
\end{theorem}

\begin{proof} 
Let $TY$ denote the algebraic tangent bundle of $Y$. From
\cite[Theorem 3.7]{MO} we know that $A$ is an algebraic group. 
The Lie algebra of $A$ is $H^{0}(Y,\, TY)$ equipped with the Lie
bracket operation of vector fields.
 
The Lie algebra of $G$ will be denoted by $\mathfrak g$.
Define $X\,:=\,\overline{{\rm PSL}(n+1, \mathbb C)}$ and 
$U\,:=\,\overline{{\rm PSL}(n+1, \mathbb C)}_{T}^{ss}(\mathcal L_{\chi}).$
The connected component, containing the identity element, of the automorphism group 
of $X$ is $G \times G$ \cite[Example 2.4.5]{Br}. From this, and the fact that
the complement $X\setminus U\, \subset\, X$ is of codimension at least three
(see Corollary \ref{cor1}), we conclude that
$$H^{0}(U,\, TU)\,=\,H^{0}(X,\, TX)\,=\, \mathfrak{g}\oplus \mathfrak{g}\, .$$
Let $\phi\,:\, U \,\longrightarrow\, Y$ be the geometric
invariant theoretic quotient map. Let 
$$T_U\,\supset\, T_{\phi}\,\longrightarrow\, U$$ be the relative tangent bundle for $\phi$.
Since $\phi$ makes $U$ a principal $T$--bundle over $Y$ 
(see Corollary \ref{cor1}(1)),
we have the following short exact sequence of vector bundles on $U$
\begin{equation}\label{ses}
0 \,\longrightarrow\, T_{\phi} \,\longrightarrow\, TU \,\longrightarrow\, 
\phi^{*}(TY) \,\longrightarrow\, 0\, ,
\end{equation}
and the relative tangent bundle 
$T_{\phi}$ is identified with the trivial vector bundle $\mathcal{O}_U\otimes_{\mathbb C}
\mathfrak{h}$,
where $\mathfrak h$ is the Lie algebra of $T$.

Set
$Z \,=\, X\setminus U.$ Since ${\rm codim}(Z) \,\geq\, 3$ (Corollary \ref{cor1}), we have 
$$H^{0}(U,\, T_{\phi})\,=\, H^{0}(X,\, \mathcal{O}_X\otimes \mathfrak{h})\,=\, \mathfrak{h}\, .$$
Note that $H^1(U,\, T_{\phi})\,=\, H^{2}_Z(X,\, \mathcal{O}_X\otimes \mathfrak{h}).$
Indeed, this follows from the following cohomology exact sequence
(see \cite[Corollary 1.9]{Gr})
$$
H^{1}(X, \,\mathcal{O}_X\otimes \mathfrak{h})\,\longrightarrow\,
H^{1}(U,\, \mathcal{O}_X\otimes \mathfrak{h})\,\longrightarrow\,
H^{2}_Z(X, \mathcal{O}_X\otimes \mathfrak{h})\,\longrightarrow\,
H^{2}(X, \,\mathcal{O}_X\otimes \mathfrak{h})
$$
combined with the fact that $H^{i}(X, \, \mathcal{O}_X) \,= \,0$ for all $i\,>\,0$ 
\cite[p. 30, Theorem]{DP}. As $X$ is smooth and ${\rm codim}(Z) \,\geq\, 3$,
it follows from \cite[Theorem 3.8 and Proposition 1.4]{Gr} that
$$H^{2}_Z(X,\, \mathcal{O}_X)\,=\, 0\, ,$$ and hence 
$H^{1}(U,\, T_{\phi})\,=\, 0.$
Now, using this fact in the long exact sequence of cohomologies
corresponding to the short exact sequence in (\ref{ses}) we obtain the
following short exact sequence:
$$0\,\longrightarrow\, 0 \oplus \mathfrak{h} \,\longrightarrow \,\mathfrak{g}\oplus 
\mathfrak{g} \,\longrightarrow\, H^{0}(U, \,\phi^{*}T{Y})\,\longrightarrow\, 0\, .$$
Hence, we have
$$H^{0}(U,\, \phi^{*}TY)\,=\,\mathfrak{g}\oplus (\mathfrak{g}/\mathfrak{h})\, .$$
By using geometric invariant theory, $H^0(Y,\, TY)$ is the invariant part
$$H^{0}(Y,\, TY)\,=\, H^{0}(U,\, \phi^{*}TY)^{\{1\}\times T}\,\subset\,
H^{0}(U,\, \phi^{*}TY)\, .
$$
Thus we have $H^{0}(Y,\, TY) \,= \,\mathfrak{g}.$ This proves (1).

To prove (2), let $\{D_{i}\,\mid\, 1\leq i \leq n\}$ be the $(G\times G)$--stable
irreducible closed subvarieties of $\overline{G}$ of codimension one such that 
$$G\,=\, \overline{G} \setminus (\bigcup_{i=1}^{n}D_{i})\, .$$ Let $D^{ss}_i
\,=\, D_i\bigcap X^{ss}\,\subset\, D_i$ be the semistable locus of $D_i.$
Set $Z\,:= \,Y\setminus(G/T),$ 
and write it as a union $$ Z\,=\, \bigcup_{i=1}^n Z_i\, ,$$
where each $Z_i \,=\, D_i^{ss}/\!\!/T$ is an irreducible closed subvariety of 
$Y$ of codimension one. As $Y$ is smooth, each $Z_i$ produces a
line bundle $L_i\,\longrightarrow\, Y$ whose pullback to $X^{ss}$
is $\mathcal{O}_{X^{ss}}(D^{ss}_i).$ Since ${\rm Pic}(X^{ss})=
{\rm Pic}(X)$ and $\{\mathcal{O}_{X}(D_{i})\}_{1\leq i \leq n}$ 
are linearly independent
in ${\rm Pic}(X)$ (see \cite[p. 26, 8.1]{DP}),
we get that  $L_i$, $1\,\leq\, i \,\leq\, n,$ are linearly independent in ${\rm Pic}(Y).$
The Picard group of $G/T$ is isomorphic to the 
group of characters of the inverse image $\widehat{T}$ of $T$ inside the 
simply connected covering $\widehat{G}$ of $G$ (see \cite{KKV}). 
Now it follows from the exact sequence in \cite[Proposition 1.8]{Fu} 
that ${\rm Pic}(Y)$ is a free Abelian group of rank $2n,$ thus completing the
proof of (2).
\end{proof}

\begin{remark}
The compactification $Y$ of $G/T$ constructed here
is an example of a non-spherical variety for the action of $G$ whose connected
component of the automorphism group is $G$.
\end{remark}

\begin{remark} Note that both $Y$ 
and $G/B\times G/B$ are smooth compactifications of $G/T$ 
with isomorphic Picard groups. Further both are Fano varieties, i.e., the
anti-canonical line bundle is 
ample. The fact that $G/B\times G/B$ is Fano is well known. That the variety $Y$ is 
Fano follows as a consequence of the exact sequence in
(\ref{ses}) together with the facts that $X$ is Fano (see \cite{DP}) and the 
codimension of $X\setminus U$ is greater than or equal to $3,$ 
where $X$ and $U$ are as in the proof
of Theorem \ref{thm1}. But $Y$ 
and $G/B\times G/B$  are not isomorphic, as ${\rm Aut}^0(Y) \simeq G$ and
${\rm Aut}^0(G/B\times G/B) \simeq G\times G,$ where ${\rm Aut}^0(M)$ denotes the
connected component of the group of algebraic automorphisms of a
smooth projective variety $M$ containing the identity element.
\end{remark}

\begin{remark}
In \cite{St}, Strickland extended the construction of $\overline{G}$ 
to any arbitrary algebraically closed field. Also, $\overline{G}$
is a Frobenius split variety in positive characteristic \cite[p. 169, 
Theorem 3.1]{St} (see \cite{MR}
for the definition of Frobenius splitting). Since $T$ is linearly reductive, using Reynolds
operator, one can see that the geometric invariant theoretic quotient of $\overline{G}$ 
for the action of $T$ is also Frobenius split for any polarization on $\overline{G}$.
\end{remark}

\section*{Acknowledgements}

We are grateful to the referee for comments to improve the exposition. The first--named author 
thanks the Institute of Mathematical Sciences for hospitality while this work was 
carried out. He also acknowledges the support of the J. C. Bose Fellowship.
The second named author would like to thank
the Infosys Foundation for the partial
support.

\end{document}